\documentclass[reqno]{amsart}

\usepackage{graphicx,subfigure}

\usepackage[normalem]{ulem}

\numberwithin{equation}{section}

\usepackage[latin1]{inputenc}
\usepackage[english]{babel}

\usepackage{amsmath,amsthm,amsfonts,latexsym,amssymb}
\usepackage[colorlinks]{hyperref}
\hypersetup{linkcolor=blue,citecolor=blue,filecolor=black,urlcolor=blue}
\usepackage{comment}

\usepackage{xcolor}

\usepackage{color,amsthm,amsfonts,mathtools}
\definecolor{darkgreen}{rgb}{0,0.7,0.1}

{ \theoremstyle{plain}
\newtheorem{theorem}{Theorem}[section]

\newtheorem{lemma}[theorem]{Lemma}
\newtheorem{corollary}[theorem]{Corollary}
  \theoremstyle{remark}
\newtheorem{remark}[theorem]{Remark}
  \theoremstyle{definition}

}

\begin{document}


\title[]{Continuous dependence for $p$-Laplace equations with varying operators}

\author[F. Colasuonno]{Francesca Colasuonno}
\address{Francesca Colasuonno\newline\indent
Dipartimento di Matematica
\newline\indent
Alma Mater Studiorum Universit\`a di Bologna
\newline\indent
piazza di Porta San Donato 5, 40126 Bologna, Italy}
\email{francesca.colasuonno@unibo.it}

\author[B. Noris]{Benedetta Noris}

\address{Benedetta Noris \newline\indent
Dipartimento di Matematica\newline\indent
Politecnico di Milano\newline\indent
Piazza Leonardo da Vinci 32, 20133 Milano, Italy}
\email{benedetta.noris@polimi.it}

\author[E. Sovrano]{Elisa Sovrano}
\address{Elisa Sovrano\newline\indent
Dipartimento di Scienze e Metodi dell'Ingegneria
\newline\indent
Universit\`a degli Studi di Modena e Reggio Emilia
\newline\indent
Via Amendola 2 - Pad. Morselli, 42122 Reggio Emilia, Italy}
\email{elisa.sovrano@unimore.it}


\begin{abstract}
For the following Neumann problem in a ball 
\[\begin{cases}
-\Delta_p u+u^{p-1}=u^{q-1}\quad&\mbox{in }B,\smallskip\\
u>0,\,u\mbox{ radial}\quad&\mbox{in }B,\smallskip\\
\dfrac{\partial u}{\partial \nu}=0\quad&\mbox{on }\partial B,
\end{cases}
\]
with $1<p<q<\infty$, we prove continuous dependence on $p$, for radially nondecreasing solutions. As a byproduct, we obtain an existence result for nonconstant solutions in the case $p\in(1,2)$ and $q$ larger than an explicit threshold.
\end{abstract}

\maketitle

\section{Introduction}
For $1 < m < \infty$, we consider the following Neumann problem
\begin{equation}\label{eq:Pq}
\tag{$P_m$}
\begin{cases}
-\Delta_m u+u^{m-1}=u^{q-1}\quad&\mbox{in }B,\smallskip\\
u>0,\,u\mbox{ radial}\quad&\mbox{in }B,\smallskip\\
\dfrac{\partial u}{\partial \nu}=0\quad&\mbox{on }\partial B,
\end{cases}
\end{equation}
where $B$ denotes the unit ball of $\mathbb R^N$ ($N\ge1$), $\nu$ is the outer unit normal of $\partial B$, and $q>m$. 

When $m=2$, this problem is the stationary version of the reaction-diffusion system for chemotaxis that was introduced by Keller and Segel in \cite{Keller}, that is
\begin{equation}\label{sys}
\begin{cases}
\dfrac{\partial v}{\partial t} = d_1\Delta v - \chi \nabla \cdot (v\nabla \ln w) \quad &\text{for } x\in \Omega, \, t>0,\smallskip \\  
\dfrac{\partial w}{\partial t} = d_2 \Delta w-a w+b v \quad &\mbox{for } x\in \Omega, \, t>0,\smallskip\\ 
\dfrac{\partial v}{\partial \nu}=\dfrac{\partial w}{\partial \nu}=0 \quad &\mbox{for } x\in \partial\Omega, \, t>0,\smallskip\\
v(x,0),\,w(x,0)>0 &\mbox{for } x\in \Omega,
\end{cases}
\end{equation}
where $\Omega$ is a bounded domain of $\mathbb R^N$ and $d_1,\,d_2,\,\chi,\,a,\,b>0$ are positive constants depending on the specific physical model. Chemotaxis is one of the simplest aggregation mechanisms for biological species: cells living in a certain environment $\Omega$ move towards higher concentrations of a chemical substance that is usually produced by the cells themselves. In the model \eqref{sys}, $v(x,t)$ stands for the cells' concentration and $w(x,t)$ for the substance's concentration; they are naturally requested to be positive quantities. Usually, organisms live in a bounded region and, from the model point of view, it is customary to suppose that there are no incomes nor outcomes from outside: zero Neumann boundary conditions guarantee mass conservation inside the region. 
It is well-known that stationary solutions of a dynamical system play a fundamental role for the comprehension of the global dynamics. In this case, stationary solutions are of the form $(v(x),w(x))$, with $v(x)=c w^{\alpha}(x)$ for some $c,\alpha>0$, and the stationary version of the system \eqref{sys} can be reduced to a scalar equation for $w$, subject to Neumann boundary conditions. More precisely, the function $u$, that is $w$ up to a multiplicative constant, solves the following elliptic problem 
\begin{equation}\label{eq}
\begin{cases}
-d^2 \Delta u+u=u^{q-1} \quad & \mbox{in }\Omega,\smallskip\\ 
u>0 &\mbox{in }\Omega,\smallskip\\
\dfrac{\partial u}{\partial \nu}=0 &\mbox{on } \partial \Omega,
\end{cases}
\end{equation}
where $d>0$ and $q>2$ are related to the constants appearing in \eqref{sys}. Existence and nonexistence of nonconstant solutions of \eqref{eq} have been widely studied since the eighties, after the paper by Lin, Ni, and Takagi \cite{LNT}. In such a paper, the authors proved that in the subcritical regime, $2<q<2N/(N-2)$, if the diffusion coefficient $d$ is sufficiently large (or, equivalently, if $|\Omega|$ is sufficiently small), the semilinear problem \eqref{eq} has only the constant solution 1, while, if $d$ is sufficiently small (or $|\Omega|$ is sufficiently large), \eqref{eq} has at least a noncostant solution. When the domain is radially symmetric, similar results have been obtained also in the supercritical regime $q> 2N/(N-2)$ for radial solutions in \cite{LN}. Finally, in the critical case $q = 2N/(N-2)$, the validity of this type of results strongly depends on the dimension $N$, cf. \cite{AY91,AY97,BKP}. We observe that the existence of nonconstant solutions of \eqref{eq} is interesting also from an application point of view, because it is heuristically related to the formation of spatially inhomogeneous patterns in large times for the dynamical system \eqref{sys}.   
More recently, it has been proved that the set of solutions of \eqref{eq} is very rich, even when confining the analysis to radial solutions: letting the exponent $q$ grow, bifurcation phenomena occur from the constant solution 1, cf. \cite{BGT}.
Beyond topological techniques, several other tools have been used to study existence and multiplicity of radial solutions of \eqref{eq}, e.g. shooting methods \cite{ABF-ESAIM,BCN3,ABF-PRE} and variational techniques \cite{BGNT,BNW,BF,BFGM,GrossiNoris,CC,CC2}.
The papers cited above include also more general nonlinearities and different operators, such as the fractional Laplacian, and the Minkowski operator.\\
The existence of nonradial solutions is a much more delicate issue, especially in the critical and supercritical regimes. In the subcritical case, Ni and Takagi in \cite{NiTakagi} prove the existence of a mountain pass solution having a unique maximum point on the boundary of the domain, provided that the diffusion coefficient $d$ in \eqref{eq} is sufficiently small. This leads to symmetry breaking in the radial setting. The critical counterpart of this result has been studied by Wang in \cite{Wang-ARMA}. In particular, when the domain is a ball, multiplicity and symmetry breaking is obtained by constructing nonradial multipeak solutions, see \cite{Wang-PRSE}. The slightly supercritical case is studied via a perturbative approach by Del Pino, Musso, and Pistoia in \cite{dPMP}, where the authors prove the existence of multipeak solutions. To our knowledge, for general supercritical nonlinearities, symmetry breaking is still an open problem. We refer to \cite{CM} for an existence result for the supercritical Neumann problem in nonradial cylindrical-type domains.
Concerning the Dirichlet counterpart of \eqref{eq}, the existence and the multiplicity of nonradial solutions have been studied via variational methods for nonsmooth functionals in invariant convex sets or via a new approach combining dynamical system techiniques with variational ones, in \cite{BCNW,CowanMoameni2023}, where solutions with axial symmetry or other multiple revolution symmetries are exhibited in annuli with sufficiently large holes. We recall that, for the Dirichlet problem, the search for nonradial solutions in a ball or in an annulus with small hole, in the critical or supercritical regime is sterile in view of the celebrated symmetry preservation result by Gidas, Ni, and Nirenberg \cite{GNN} and of the more recent result by Grossi, Pacella, and Yadava \cite{GPY}, respectively. On the other hand, it is known that for expanding annuli of radii $R$ and $R+k$, the number of nonradial solutions grows to infinity as $k\to\infty$, see \cite{Coffman,Li,CW,ByeonKimPistoia2013}. Moreover, we refer to \cite{BCNW2} for multiplicity and symmetry breaking results for the Dirichlet problem set in an exterior domain, where the lack of compactness is even more severe due to the unboundedness of the domain. 

Beyond the study of existence, and symmetry properties of solutions, the asymptotic analysis has revealed to be a powerful tool to get insights into the properties of solutions of \eqref{eq}, cf. \cite{BGNT,LNT}. For instance, in \cite{BGNT} an asymptotic analysis was used to prove uniqueness and nondegeneracy of solutions.
In \cite{BF_NA}, we exploited the asymptotic behavior of a least energy radially nondecreasing solution of ($P_m$) with $1<m<2$, as the exponent $q$ in the nonlinearity goes to infinity, to distinguish it from the constant solution 1. Moreover, in the same paper, we detected the asymptotic behavior, with respect to the growing power $q$, of the higher energy radially nondecreasing solution of ($P_m$) previously found in \cite{BFGM}.

In the present paper, we prove continuous dependence on $m$ of certain classes of solutions of \eqref{eq:Pq} for $m> 1$. We note that the operator changes with the parameter $m$. Such a continuity allows us to get an existence result for $m<2$ in a left neighbourhood of $2$, providing the existence of nonconstant solutions for \eqref{eq:Pq} when $q$ is larger than an explicit threshold.

Problem \eqref{eq:Pq} has been widely studied for a fixed $m\in(1,\infty)$. Moreover, when $q>m$ is a possibly Sobolev supercritical exponent, it has been proved in \cite{BGNT,ABF-ESAIM,ABF-PRE} that the set of solutions is very rich. Here we restrict our analysis to radially nondecreasing solutions. To this aim, we introduce the cone of nonnegative, radial, radially nondecreasing $W^{1,m}$-functions
\begin{equation}\label{cone}
\mathcal C_m:=\{u\in W^{1,m}_{\mathrm{rad}}(B)\,:\, u\ge0,\,u(r_1)\le u(r_2) \mbox{ for all }0<r_1\le r_2\le1\},
\end{equation}
where with abuse of notation we write $u(|x|):=u(x)$. 

Our first continuity result states that any sequence of $\mathcal C_{p_n}$-solutions of $(P_{p_n})$, with $p_n\to p$, admits a limit which is in turn a solution of $(P_p)$. 
\begin{theorem}\label{thm:upntou_p}
Let $(p_n)\subset(1,\infty)$ be such that $p_n\to p\in(1,\infty)$ as $n\to \infty$. 
Let $q>p$ and $u_{p_n}\in \mathcal C_{p_n}$ be a solution of \eqref{eq:Pq} with $m=p_n$ for every $n\in\mathbb N$. Then, up to a subsequence, $u_{p_n}\to u_{p}$ in $W^{1,p}(B)\cap C^{0,\beta}(\bar B)$ for every $\beta\in (0,1)$, where $u_p\in \mathcal C_{p}$ is a solution of \eqref{eq:Pq} with $m=p$.
\end{theorem}

Our second result concerns, among $\mathcal C_m$-solutions, those with minimal energy. We remark that, in the Sobolev supercritical setting, the energy functional associated to \eqref{eq:Pq} is not well-defined in $W^{1,m}(B)$, so it is necessary to specify what we mean by minimal energy solutions. 
Taking advantage of the fact that the cone $\mathcal C_m$ is embedded in $L^\infty(B)$, cf. \cite[Lemma 2.2]{BF}, 
it is possible to associate to \eqref{eq:Pq} an energy functional $I_m$ which is well-defined in $W^{1,m}(B)$ and such that its critical points belonging to the cone $\mathcal C_m$ are weak solutions of \eqref{eq:Pq}, see \eqref{eq:energy_I_m_def} ahead. This allows to investigate the existence of solutions to \eqref{eq:Pq} via variational methods inside $\mathcal C_m$;
as $I_m$ is unbounded from below, minimal energy solutions are defined as minima over a Nehari-type set $\mathcal N_m$, see \eqref{eq:N_q_def}.
The existence of a nonconstant radially nondecreasing solution that achieves the infimum of $I_m$ over $\mathcal N_m$
has been proved for $q$ sufficiently large. 
Hereafter we shall refer to such solutions as {\it $\mathcal C_m$-ground state solutions}.
When $1<m<2$, for sufficiently large values of $q$, problem \eqref{eq:Pq} admits two distinct nonconstant solutions in $\mathcal C_m$, cf. \cite{BFGM,ABF-ESAIM}. These solutions have been studied variationally in \cite{BFGM}, and they can be distinguished by their energy: one is a $\mathcal C_m$-ground state solution, the other is a mountain pass solution over $\mathcal N_m$ and so has higher energy.
The next theorem states that $\mathcal C_{p_n}$-ground state solutions converge to a $\mathcal C_p$-ground state solution as $p_n\to p$. 

\begin{theorem}\label{thm:upntou_p-GS}
Let $(p_n)\subset(1,\infty)$ be such that $p_n\to p\in(1,\infty)$ as $n\to \infty$.  
Let $q>p$ and $u_{p_n}$ be a $\mathcal C_{p_n}$-ground state solution for \eqref{eq:Pq} with $m=p_n$ for every $n\in\mathbb N$. Then, up to a subsequence, $u_{p_n}\to u_{p}$ in $W^{1,p}(B)\cap C^{0,\beta}(\bar B)$ for every $\beta\in (0,1)$, where $u_{p}$ is a $\mathcal C_{p}$-ground state solution of \eqref{eq:Pq} with $m=p$.
\end{theorem}

In order to ensure the convergence of the full sequence in Theorem \ref{thm:upntou_p-GS}, the uniqueness of the limit is needed. The only case where uniqueness has been proved is for $m=2$ and $q$ sufficiently large, cf. \cite{BGNT}. 
According to numerical simulations in \cite{ABF-ESAIM}, uniqueness seems likely to hold also for $m\neq 2$, but a proof of this is not yet available.  
In view of uniqueness of the $\mathcal C_2$-ground state solution of $(P_2)$, the previous theorem reads as follows. 
\begin{corollary}\label{cor:uptou2}
Let $(p_n)\subset(1,\infty)$ be such that $p_n\to 2$ as $n\to \infty$ and let $u_{p_n}$ be a $\mathcal C_{p_n}$-ground state solution for \eqref{eq:Pq} with $m=p_n$ for every $n\in\mathbb N$. Then, $u_{p_n}\to u_2$ in $H^1(B)\cap C^{0,\nu}(\bar B)$ for every $\nu\in (0,1)$, where $u_2$ is the unique $\mathcal C_2$-ground state solution of \eqref{eq:Pq} with $m=2$.
\end{corollary}
Although a $\mathcal C_m$-ground state solution exists for every $q>m$, cf. \cite{BNW,BF}, in some cases it coincides with the trivial constant solution 1, see e.g. \cite[Proposition 4.1]{BNW}.  
On the other hand, the $\mathcal C_m$-ground state solution is nonconstant: 
\begin{itemize}
\item[(i)] when $m=2$, if $q>2+\lambda_2^{\mathrm{rad}}$, see \cite{BNW}, where $\lambda_2^{\text{rad}}$ the first {\it positive} eigenvalue, associated to a radial eigenfunction, of the Laplacian in the ball $B$ under Neumann boundary conditions. More precisely, $\lambda_2^{\text{rad}}$ is related to the following eigenvalue problem
\begin{equation}\label{eq:eigenvalues}
\begin{cases}
-\Delta u=\lambda u\quad&\mbox{in }B,\smallskip\\
u\mbox{ radial}&\mbox{in }B,\smallskip\\
\dfrac{\partial u}{\partial \nu}=0&\mbox{on }\partial B
\end{cases}
\end{equation}
(we use the subscript 2 in $\lambda_2^{\mathrm{rad}}$ because the first eigenvalue is zero);\smallskip
\item[(ii)] when $m>2$, if $q>m$, see \cite{BF};\smallskip
\item[(iii)] when $1<m<2$, if $q$ is {\it sufficiently} large, see \cite{BFGM}.
\end{itemize}
In the above cases, Corollary \ref{cor:uptou2} is nontrivial. We highlight that an explicit threshold value of $q$ for the nonconstancy of the $\mathcal C_m$-ground state solution when $1<m<2$ is not available. From (i), in view of the continuity in $m$ stated in Theorem \ref{thm:upntou_p-GS}, we deduce the following existence result.
\begin{corollary}\label{cor:existence_m<2}
Let $q>2+\lambda_2^{\mathrm{rad}}$ and $1<m<2$ be sufficiently close to $2$. If $u_m$ is a $\mathcal C_m$-ground state solution of \eqref{eq:Pq}, then $u_m$ is nonconstant.    
\end{corollary}

We observe that the result stated in the above corollary can be rephrased in terms of the radius of the ball in the following sense. Throughout the paper, we set our problem in the unit ball $B$ and make assumptions on the exponents $q$ and $m$. Via the usual change of variable $\tilde u(x)=u(x/R)$, we can set the problem in the ball $B_R$ of radius $R$, and the condition on $q$ to get the existence reads as 
$q>2+\lambda_2^{\mathrm{rad}}(R)$, where now $\lambda_2^{\mathrm{rad}}(R)$ is the second radial eigenvalue of the Neumann Laplacian in the ball $B_R$. Since, by scaling, $\lambda_2^{\mathrm{rad}}(R)\to 0$ as $R\to\infty$, the assumption $q>2+\lambda_2^{\mathrm{rad}}(R)$ can be satisfied either increasing the exponent $q$ in the nonlinearity or increasing the radius $R$ of the ball where the problem is set.
\smallskip

We show below some numerical simulations obtained with the software Wolfram Mathematica \cite{Mathematica}. We consider radial solutions to problem \eqref{eq:Pq}; for simplicity, we confine the numerical analysis to the dimension $N=1$, namely
\begin{equation}\label{eq:modelloN1}
\begin{cases}
-(|u'|^{p-2}u')'+u^{p-1}=u^{q-1} \quad &\text{in } (0,1) \\
u>0 \quad &\text{in } (0,1) \\
u'(0)=u'(1)=0.
\end{cases}
\end{equation}

\begin{figure}[h!t]
\begin{center}
\subfigure[$q=30$\label{fig:1-a}]{\includegraphics[height=3.7cm]{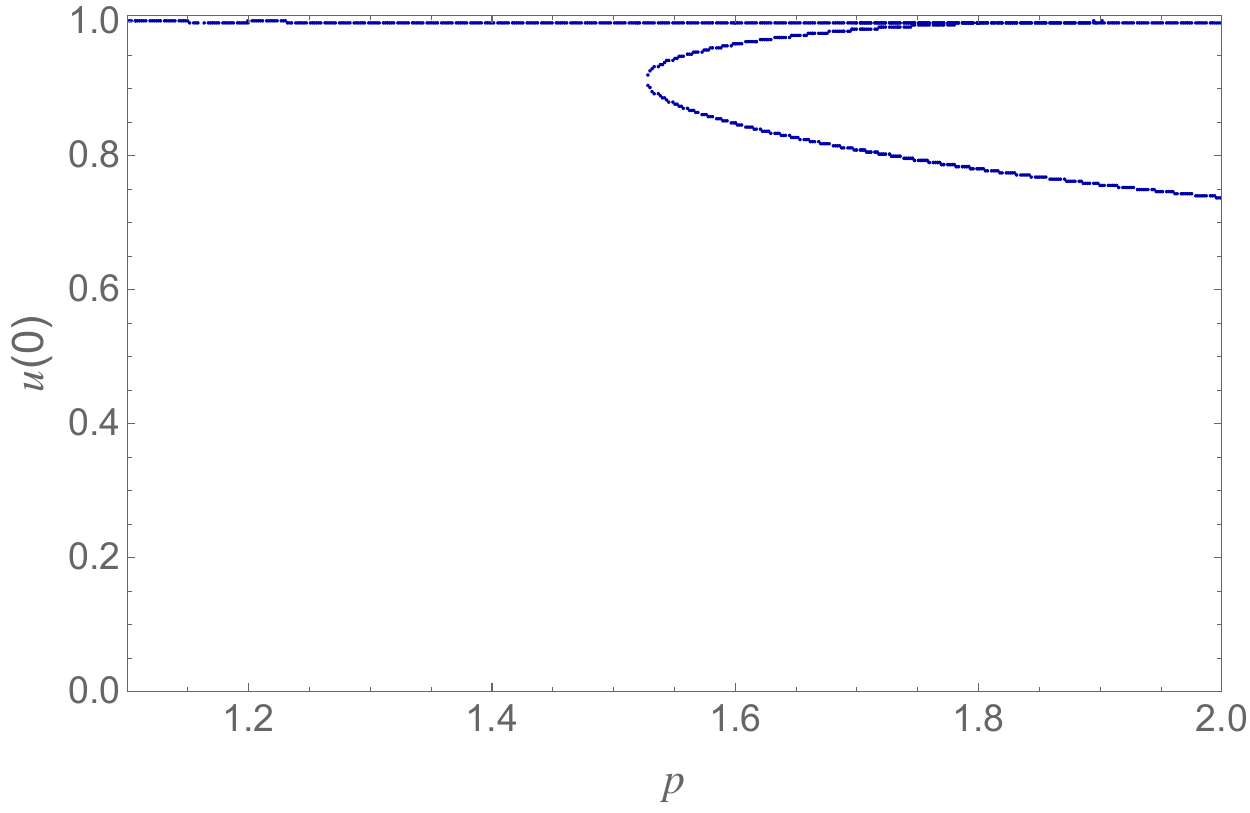}}
\subfigure[$q=60$\label{fig:1-b}]{\includegraphics[height=3.7cm]{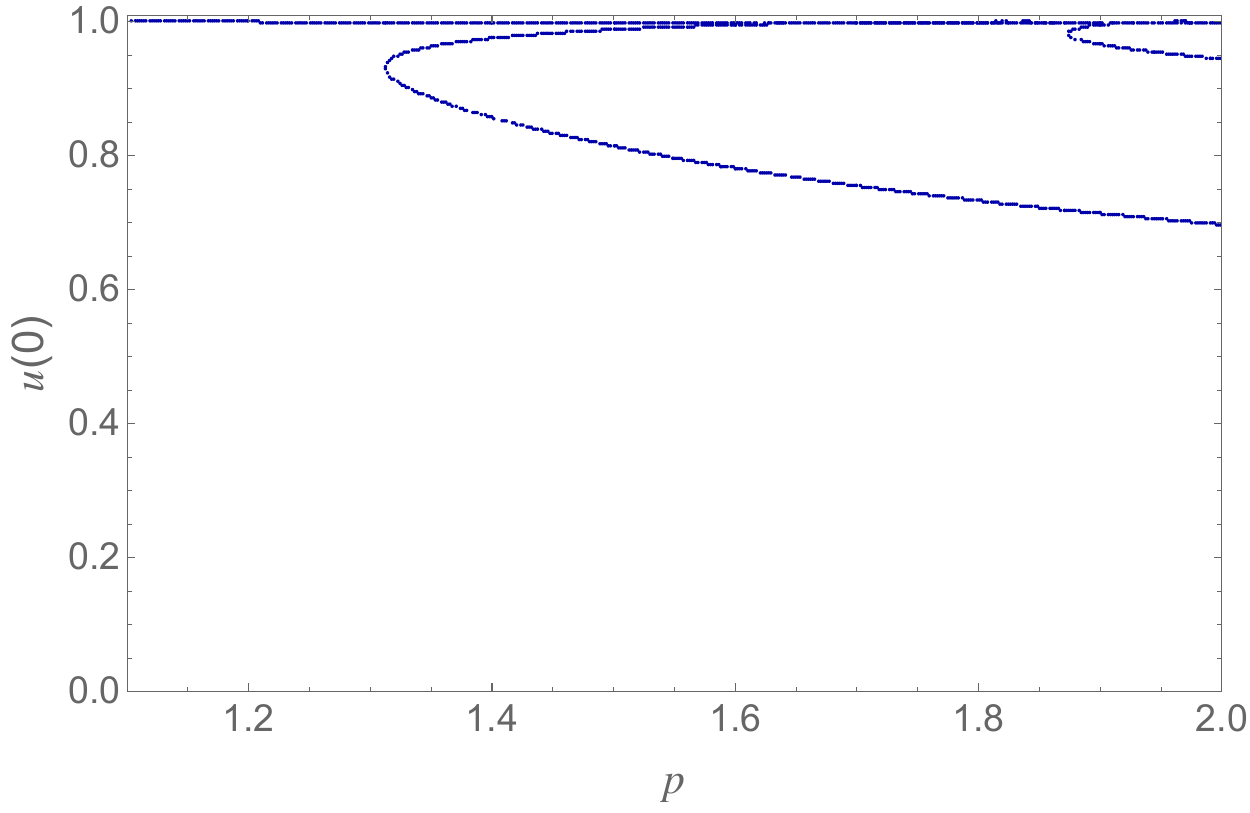}}
\caption{Bifurcation diagrams for problem \eqref{eq:modelloN1} as a function of $p$.\label{fig:1}}
\end{center}
\end{figure}

Figure~\ref{fig:1} presents two bifurcation diagrams associated to \eqref{eq:modelloN1}, when the parameter $p$ ranges between $p=1.1$ and $p=2$, and the value of $q$ is fixed, respectively equal to $30$ in Figure~\ref{fig:1-a} and to $60$ in Figure~\ref{fig:1-b}.
We plot the values of $p$ in the horizontal axis versus the values of $u(0)$ in the vertical one, where $u$ solves~\eqref{eq:modelloN1}. 
Increasing $p$, we exhibit numerical evidence of the presence of positive non-constant solutions to problem \eqref{eq:modelloN1}.
The diagram suggests that, for $q=30$, the branch starts from a turning point (for $p\gtrsim 1.5$), whereas, for $q=60$, the two branches start from two turning points (respectively, for $p\gtrsim 1.3$ and for $p\gtrsim 1.85$).

We observe that the eigenvalue problem introduced in \eqref{eq:eigenvalues}, in this one-dimensional setting reads as
\[
\begin{cases}
-u''=\lambda u\quad\mbox{in }(0,1),\\
u'(0)=u'(1)=0,
\end{cases}
\]
are $\lambda_n=(n-1)^2\pi^2$ for $n\ge 2$. In particular, $q=30\in (2+\lambda_2,2+\lambda_3)$ and $q=60\in(2+\lambda_3, 2+\lambda_4)$. Thus, for $p$ close to two, the appearance of the first branch of solutions for $q=30$ is coherent with the statement of Corollary \ref{cor:existence_m<2}. The appearance of a second branch of solutions for $q=60$ is in agreement with the multiplicity results proved in \cite{ABF-ESAIM}.

The paper is organized as follows. In Section \ref{sec:2} we prove Theorem \ref{thm:upntou_p} and in Section \ref{sec:3} we prove Theorem \ref{thm:upntou_p-GS} and its consequences.

\section{Proof of Theorem \ref{thm:upntou_p}}\label{sec:2}
Throughout the paper, for radial functions we use alternatively $u(x)$ and $\nabla u(x)$, with $x\in B$, or $u(r)$ and $u'(r)$, with $r=|x|\in(0,1)$, with abuse of notation. 

Let $(p_n)$ be such that $p_n\to p$ and let $u_{p_n}\in \mathcal C_{p_n}$ be a solution of \eqref{eq:Pq} with $m=p_n$, as in the statement of Theorem~\ref{thm:upntou_p}.
We split the proof of this theorem into two steps: first, we prove the existence of a limit profile of the sequence $(u_{p_n})$, then we show that the limit function $u_p$ is a solution of the problem with $m=p$.

For the first step, a priori estimates are crucial. The main advantage for working in the cone $\mathcal C_m$ introduced in \eqref{cone} is that functions belonging to the cone are bounded. Indeed, $\mathcal C_m$ is continuously embedded in $L^\infty(B)$, namely
\begin{equation}\label{eq:CNm}
\|u\|_{L^\infty(B)}\le C(N,m)\|u\|_{W^{1,m}(B)}\quad\mbox{for all }
u\in\mathcal{C}_m,
\end{equation}
for some constant $C(N,m)>0$, see e.g. \cite[Lemma 2.2]{BF}.
Moreover, we note that, by standard elliptic regularity, every solution of the problem \eqref{eq:Pq} is $C^{1,\alpha}(\bar B)$ for some $\alpha\in (0,1)$, cf. \cite{dibenedetto1982c1+,tolksdorf1984regularity,Lieberman}. 

We report below $C^1$-a priori estimates for solutions of \eqref{eq:Pq} belonging to $\mathcal C_m$ proved via a phase plane analysis.
\begin{lemma}[{\cite[Lemma~5.5]{BF}} and {\cite[Lemma~2]{BF_NA}}]\label{commonbound}
Let $1<m<\infty$, $q>m$, and $u\in \mathcal C_m$ be a solution of \eqref{eq:Pq}. For every $r\in [0,1]$ it holds 
\[
u(r) \leq \left(\frac{q}{m}\right)^\frac{1}{q-m} 
\quad\textrm{ and }\quad
u'(r) \leq \left(\frac{q-m}{q(m-1)}\right)^\frac{1}{m}.
\]
\end{lemma}

We are now ready to prove the existence of a limit profile of $(u_{p_n})$. 
\begin{lemma}\label{le:limit_profile}
Let $(p_n)\subset(1,\infty)$ be such that $p_n\to p\in(1,\infty)$ as $n\to \infty$. Let $q>p$ and $u_{p_n}\in \mathcal C_{p_n}$ be a solution of \eqref{eq:Pq} with $m=p_n$ for every $n\in\mathbb N$. Then, there exists $u_p\in \mathcal C_{p}$ such that, up to a subsequence,
\[
\begin{gathered}
u_{p_n}\rightharpoonup u_{p} \quad\mbox{in } W^{1,s}(B) \mbox{ for every }s\in (1,\infty)\quad\mbox{ and}\\
u_{p_n}\to u_{p}\quad\mbox{in } C^{0,\beta}(\bar B) \mbox{ for every }\beta\in (0,1).
\end{gathered}
\] 
\end{lemma}
\begin{proof}
By Lemma \ref{commonbound}, eventually in $n$
\begin{equation}\label{eq:unif_bound}
u_{p_n}(r) \le \left(\frac{q}{p}\right)^\frac{1}{q-p}+1 
\quad\textrm{ and }\quad
u_{p_n}'(r) \le  \left(\frac{q-p}{q(p-1)}\right)^\frac{1}{p}+1.
\end{equation}
Hence, by the Arzel\`a-Ascoli Theorem, for a renamed subsequence $(u_{p_n})$, $u_{p_n}\to u_p$ in $L^\infty(B)$. Then, for every $\beta\in (0,1)$, and for every $x,\,y\in\bar B$ and $\gamma\in(\beta,1)$,
\[
\begin{aligned}
&\frac{|(u_{p_n}-u_p)(x)-(u_{p_n}-u_p)(y)|}{|x-y|^\beta}\\
&\hspace{.5cm}=\left(\frac{|(u_{p_n}-u_p)(x)-(u_{p_n}-u_p)(y)|}{|x-y|^\gamma}\right)^{\frac{\beta}{\gamma}}|(u_{p_n}-u_p)(x)-(u_{p_n}-u_p)(y)|^{1-\frac{\beta}{\gamma}}\\
&\hspace{.5cm}\le |u_{p_n}-u_p|_{C^{0,\gamma}(\bar B)}^{\frac{\beta}{\gamma}} 2\|u_{p_n}-u_p\|_{L^\infty(B)}^{1-\frac{\beta}{\gamma}} \longrightarrow 0	\quad\mbox{ as }n\to\infty,
\end{aligned}
\]
in view of the embedding $C^1(\bar B)\hookrightarrow C^{0,\gamma}(\bar B)$, see for instance \cite[Lemma 6.33]{GilbargTrudinger}. In conclusion, for a renamed subsequence $(u_{p_n})$, $u_{p_n}\to u_p$ in $C^{0,\beta}(\bar B)$ for all $\beta\in(0,1)$.

As a consequence, $u_{p_n}\to u_p$ pointwise, this implies that $u$ is nonnegative, radial, and radially nondecreasing. Moreover, from \eqref{eq:unif_bound} we infer also that $(u_{p_n})$ is bounded in $W^{1,s}(B)$ for every $s\in (1,\infty)$ and so, up to a subsequence, $u_{p_n}\rightharpoonup u_p$ in $W^{1,s}(B)$ for every $s\in (1,\infty)$. In particular, $u_p\in W^{1,p}(B)$, thus $u_p\in \mathcal C_p$.  
\end{proof}

In order to prove that the limit profile $u_p$ in the previous lemma is a weak solution of \eqref{eq:Pq} with $m=p$, we need the following characterization of weak solutions. 
\begin{lemma}[{\cite[Lemma 3]{BF}}]\label{lem:conseq-def-equiv}
Let $1<m<\infty$ and $q>m$. A function $u\in \mathcal C_m$ is a weak solution of \eqref{eq:Pq} if and only if, for every $\varphi\in W^{1,m}(B)$,
\begin{equation}\label{eq:def-sol-equiv}
\int_B\frac{|\nabla u|^m+u^m}{m}\, dx\le \int_B\frac{|\nabla \varphi|^m+|\varphi|^m}{m}\,dx - \int_B u^{q-1}(\varphi-u)\, dx.  
\end{equation}
\end{lemma}

\begin{lemma}\label{lem:conv-integrals-2}
Let $(p_n)\subset(1,\infty)$ be such that $p_n\to p\in(1,\infty)$ as $n\to \infty$. Let $q>p$ and $(u_{p_n})$ be the renamed subsequence converging to the limit profile $u_p$, as in Lemma \ref{le:limit_profile}.
Then, the following limits hold
\begin{align}
\lim_{n\to \infty}\int_B \frac{u_{p_n}^{p_n}}{p_n} dx&=\int_B \frac{u_{p}^{p}}{{p}} dx,\label{eq:p}\smallskip\\
\lim_{n\to \infty}\int_B u_{p_n}^{q} dx&=\int_B u_{p}^{q} dx,\label{eq:q}\smallskip\\
\lim_{n\to \infty}\int_B \frac{|\nabla u_{p_n}|^{p_n}}{p_n} dx &= \int_B \frac{|\nabla u_{p}|^{p}}{p} dx. \label{eq:grad}
\end{align}
\end{lemma}
\begin{proof}
By Lemma \ref{le:limit_profile}, $(u_{p_n})$ converges to $u_p$ a.e. in $B$; moreover, by Lemma \ref{commonbound}, 
\[
\frac{|u_{p_n}(r)|^{p_n}}{p_n} \le \frac1p\left(\frac{q}{p}\right)^\frac{p}{q-p}+1,
\]
eventually in $n$. Hence \eqref{eq:p} follows by dominated convergence.
We can argue similarly to prove \eqref{eq:q}.

The proof of \eqref{eq:grad} is more delicate and requires a different approach.  
Since $u_{p_n}$ solves ($P_{p_n}$) for every $n$, inequality \eqref{eq:def-sol-equiv} rewrites as
\begin{equation}\label{eq:ineq-var-p}
\int_B\frac{|\nabla u_{p_n}|^{p_n}+u_{p_n}^{p_n}}{p_n}\,dx\le \int_B\frac{|\nabla \varphi|^{p_n}+|\varphi|^{p_n}}{p_n}\,dx - \int_B u_{p_n}^{q-1}(\varphi- u_{p_n})\,dx
\end{equation}
for every $\varphi\in W^{1,p_n}(B)$. 
By density, it holds for every $\varphi\in W^{1,p}(B)$.

On the one hand, passing to the superior limit in the previous inequality and using \eqref{eq:p} and \eqref{eq:q}, we have that
\[
\begin{aligned}
\limsup_{n\to \infty}\int_B&\frac{|\nabla u_{p_n}|^{p_n}}{p_n}\,dx\\
&\le \int_B\frac{|\nabla \varphi|^p+|\varphi|^p}{p}\,dx - \int_B u_p^{q-1}(\varphi- u_p)\,dx -\int_B\frac{u_p^p}{p}\,dx
\end{aligned}
\]
for every $\varphi\in W^{1,p}(B)$. 
In particular, specifying the previous inequality for $\varphi=u_p$, we obtain
\begin{equation}\label{eq:limsup-p}
\limsup_{n\to \infty}\int_B \frac{|\nabla u_{p_n}|^{p_n}}{p_n}\,dx\\
\le \int_B\frac{|\nabla u_p|^p\,dx}{p}.
\end{equation}

On the other hand, let $\varepsilon\in(0,p-1)$. We know from Lemma \ref{le:limit_profile} that $u_{p_n} \rightharpoonup u_p$ weakly in $W^{1,p-\varepsilon}(B)$. As  the map
\[
u \in W^{1,p-\varepsilon}(B) \mapsto \int_B |\nabla u|^{p-\varepsilon}\,dx
\]
is continuous and convex, it is weakly lower semicontinuous, hence
\begin{equation}\label{eq:liminf-grad}
\int_B|\nabla u_p|^{p-\varepsilon}\,dx\le \liminf_{n\to\infty}\int_B |\nabla u_{p_n}|^{p-\varepsilon}\,dx.
\end{equation}
Being $p_n>p-\varepsilon$ eventually in $n$, it is possible to apply the H\"older inequality with exponents $p_n/(p-\varepsilon)$ and $p_n/(p_n-p+\varepsilon)$, thus obtaining
\[
\begin{aligned}\left(\int_B |\nabla u_{p}|^{p-\varepsilon}\,dx\right)^{\frac{p}{p-\varepsilon}} 
&\leq \liminf_{n\to\infty}\left( \int_B |\nabla u_{p_n}|^{p_n}\,dx \right)^\frac{p}{p_n} |B|^\frac{p(p_n-p+\varepsilon)}{p_n(p-\varepsilon)}\\
&=|B|^{\frac{\varepsilon}{p-\varepsilon}}\liminf_{n\to\infty}\int_B |\nabla u_{p_n}|^{p_n}\,dx.
\end{aligned}
\]
Since by dominated convergence, 
\[
\lim_{\varepsilon\to 0^+}\int_B|\nabla u_p|^{p-\varepsilon}\,dx = \int_B|\nabla u_p|^{p}\,dx,
\] 
we get from the previous estimate
\begin{equation}\label{eq:liminf-p}
\begin{aligned}
\int_B|\nabla u_p|^{p}\,dx &=\lim_{\varepsilon\to 0^+}|B|^{-\frac{\varepsilon}{p-\varepsilon}}\left(\int_B |\nabla u_{p}|^{p-\varepsilon}\,dx\right)^{\frac{p}{p-\varepsilon}} \\
&\leq \liminf_{n\to\infty}\int_B |\nabla u_{p_n}|^{p_n}\,dx.
\end{aligned}
\end{equation}
By combining \eqref{eq:liminf-p} with \eqref{eq:limsup-p}, we obtain
the desired convergence of the gradient terms, thus concluding the proof. 
\end{proof}
\smallskip

\begin{proof}[Proof of Theorem \ref{thm:upntou_p}]
Let $(u_{p_n})$ be the renamed subsequence converging weakly to the limit profile $u_p$,  as in Lemma~\ref{le:limit_profile}. 
Let us prove that $u_p$ solves ($P_p$).  By Lemma~\ref{lem:conseq-def-equiv}, it is enough to show that, for every $\varphi \in W^{1,p}(B)$,
\begin{equation}\label{eq:u_p-sol}
\int_B\frac{|\nabla u_p|^p+u_p^p}{p}\, dx\le \int_B\frac{|\nabla \varphi|^p+|\varphi|^p}{p}\,dx - \int_B u_p^{q-1}(\varphi-u_p)\, dx.  
\end{equation}
Since, for every $n$, $u_{p_n}$ is a weak solution of \eqref{eq:Pq}, it holds
\[
\int_B\frac{|\nabla u_{p_n}|^{p_n}+u^{p_n}}{p_n}\, dx\le \int_B\frac{|\nabla \varphi|^{p_n}+|\varphi|^{p_n}}{p_n}\,dx - \int_B u_{p_n}^{q-1}(\varphi-u_{p_n})\, dx,
\]
for every $\varphi \in C^1(\bar{B})$.  By \eqref{eq:p} and \eqref{eq:grad}, we have
\[
\int_B\frac{|\nabla u_{p_n}|^{p_n}+u^{p_n}}{p_n}\, dx
\to
\int_B\frac{|\nabla u_p|^p+u_p^p}{p}\, dx
\]
as $n\to\infty$. Moreover, by dominated convergence,
\[
\lim_{n\to\infty}\int_B\frac{|\nabla \varphi|^{p_n}+|\varphi|^{p_n}}{p_n}\,dx
= \int_B\frac{|\nabla \varphi|^p+|\varphi|^p}{p}\,dx.
\]
Finally, by \eqref{eq:q} and again dominated convergence,
\[
\lim_{n\to\infty} \int_B u_{p_n}^{q-1}(\varphi-u_{p_n})\, dx =
 \int_B u_p^{q-1}(\varphi-u_p)\, dx.  
\]
Hence \eqref{eq:u_p-sol} holds and $u_p$ solves ($P_p$). 

It remains to prove the strong convergence.
By the compactness of the embedding $W^{1,p}(B)\hookrightarrow L^p(B)$,  we have
\begin{equation}\label{eq:Lp-strong-conv}
\lim_{n\to\infty} \|u_{p_n}\|_{L^p(B)}=\|u_{p}\|_{L^p(B)}.
\end{equation}
Moreover, we have
\begin{equation}\label{eq:Wp-strong-conv}
\begin{split}
\left| \|\nabla u_{p_n} \|_{L^p(B)}^p - \|\nabla u_{p} \|_{L^p(B)}^p \right|
\leq &
\left| \|\nabla u_{p_n} \|_{L^p(B)}^p - \|\nabla u_{p_n} \|_{L^p(B)}^{p_n} \right| + \\
& \hspace{.5cm}\left| \|\nabla u_{p_n} \|_{L^p(B)}^{p_n} - \|\nabla u_{p} \|_{L^p(B)}^p \right| \longrightarrow 0.
\end{split}
\end{equation}
Indeed, the second addend on the right hand side converges to $0$ as $n\to\infty$ by virtue of \eqref{eq:grad}; the first addend tends to $0$ by dominated convergence, in view of Lemma \ref{commonbound}. 

From \eqref{eq:Lp-strong-conv} and \eqref{eq:Wp-strong-conv} we infer that
\[
\lim_{n\to\infty} \|u_{p_n}\|_{W^{1,p}(B)}=\|u_{p}\|_{W^{1,p}(B)}.
\]
Finally,  in light of the weak convergence $u_{p_n} \rightharpoonup u_p$ in $W^{1,p}(B)$, the last expression readily implies
\[
\lim_{n\to\infty} \|u_{p_n} - u_p \|_{W^{1,p}(B)} =0,
\]
thus concluding the proof.
\end{proof}

\section{Proof of Theorem \ref{thm:upntou_p-GS}}\label{sec:3}

Let us describe the variational setting introduced in \cite{BNW,BF,BFGM} to find solutions of \eqref{eq:Pq} in $\mathcal C_m$. 
We introduce the modified nonlinearity
\begin{equation}\label{eq:phi_s0}
f_q(s):=\begin{cases}s^{q-1}\quad&\mbox{if }s\in[0,s_0],\smallskip\\
s_0^{q-1}+\dfrac{q-1}{\ell-1}s_0^{q-\ell}(s^{\ell-1}-s_0^{\ell-1})&\mbox{if } s\in(s_0,\infty),\end{cases}
\end{equation}
with $\ell\in (m,m^*)$, $m^*$ being the critical Sobolev exponent and 
\begin{equation}\label{eq:s_0_def}
 s_0=s_0(m):=\left(\frac{q}{m}\right)^{\frac{1}{q-m}}+1.
 \end{equation}
In view of the $L^\infty$-estimate given in Lemma \ref{commonbound}, it holds that every solution of the modified problem
\begin{equation}\label{eq:f_tilde_q}
\begin{cases}
-\Delta_m u+ u^{m-1}={f}_q(u)\quad&\mbox{in }B,\smallskip\\
u>0&\mbox{in }B,\smallskip \\
\dfrac{\partial u}{\partial \nu}=0&\mbox{on }\partial B,
\end{cases}
\end{equation}
belonging to $\mathcal C_m$ also solves the original problem \eqref{eq:Pq}. 
Hence, when looking for solutions in $\mathcal C_m$, it is possible to associate to equation \eqref{eq:Pq} an energy functional that is well defined in $W^{1,m}(B)$, namely
\begin{equation}\label{eq:energy_I_m_def}
I_m (u):= \int_B\left(\frac{|\nabla u|^m}{m}+\frac{|u|^m}{m}-F_q(u) \right)dx,
\end{equation}
where $F_q(u):=\int_0^u f_q(s)ds$.  Although $I_m$ is not the standard energy functional associated to \eqref{eq:Pq}, it has the property that its critical points belonging to the cone $\mathcal C_m$ are weak solutions of \eqref{eq:Pq}. In particular, if $u\in\mathcal{C}_m$ solves \eqref{eq:Pq}, by Lemma \ref{commonbound} $\|u\|_{L^\infty(B)}< s_0$, and so its energy reads as 
\begin{equation}\label{eq:Isol}
I_m (u)= \int_B\left(\frac{|\nabla u|^m}{m}+\frac{u^m}{m}-\frac{u^q}{q} \right)dx=\frac{1}{m}\|u\|^m_{W^{1,m}(B)}-\frac{1}{q}\|u\|^q_{L^q(B)}.
\end{equation}
The truncation of the nonlinearity beyond the $L^\infty$-bound of $\mathcal C_m$-solutions described above allows to investigate the existence of solutions to \eqref{eq:Pq} via variational methods inside $\mathcal C_m$.

In \cite{BNW,BF,BFGM}, the following Nehari-type set inside $\mathcal C_m$ is defined 
\begin{equation}\label{eq:N_q_def}
\mathcal N_m:=\left\{u\in\mathcal C_m \setminus\{0\}\,: 
\int_B(|\nabla u|^m+|u|^m)dx=\int_B f_q(u)u\,dx\right\}.
\end{equation}
We recall that the Nehari-type set $\mathcal N_m$ has the property that every nonzero point of $\mathcal C_m$ can be uniquely projected onto $\mathcal N_m$, in the sense expressed by the following lemma.
\begin{lemma}[{\cite[{Lemma 5.3}]{BF}}]\label{gH} For every $u\in\mathcal C_m\setminus\{0\}$ there exists a unique positive number $h_m(u)$ such that $h_m(u)u\in\mathcal{N}_m$. In particular, if $u\in \mathcal C_m\setminus\{0\}$ has $\|u\|_{L^\infty(B)}\le s_0$, it results 
\[
h_m(u)=\frac{\|u\|^m_{W^{1,m}(B)}}{\|u\|_{L^q(B)}^q},
\]
by \eqref{eq:phi_s0}. Finally, $h_m(u)=1$ if and only if $u\in \mathcal N_m$. 
\end{lemma}
We recall that there exists a nonconstant solution $u\in\mathcal C_m$ of \eqref{eq:Pq} that achieves the critical level
\begin{equation}\label{eq:c_q_p_def}
c_m := \inf_{u\in \mathcal N_m } I_m (u),
\end{equation}
provided that $q>m$ when $m>2$, $q>2+\lambda_2^{\text{rad}}$ for $m=2$ (where $\lambda_2^{\text{rad}}$ is the first nonzero eigenvalue of the Laplacian in the ball $B$ under Neumann boundary conditions), and $q$ sufficiently large for $1<m<2$.

We observe that under the assumptions of Theorem~\ref{thm:upntou_p}, by \eqref{eq:Isol}, the convergences in Lemma \ref{lem:conv-integrals-2} imply that 
\begin{equation}\label{eq:conv_en}
\lim_{n\to\infty}I_{p_n}(u_{p_n})= I_p(u_p).
\end{equation} 
It is then reasonable to expect that ground state solutions of ($P_{p_n}$) converge to a ground state solution of ($P_p$). We are now ready to prove Theorem \ref{thm:upntou_p-GS}.
\begin{proof}[Proof of Theorem \ref{thm:upntou_p-GS}] The proof technique of this theorem is inspired by \cite{GrossiNoris}. In view of Theorem \ref{thm:upntou_p}, it remains to prove that $u_p$ has minimal energy over $\mathcal N_p$, that is 
\[
I_p(u_p)=c_p:=\inf_{u\in\mathcal N_p}I_p(u).
\] 
Since $u_p\in\mathcal C_p$ solves ($P_p$), it belongs to $\mathcal N_p$, and so $c_p\le I_p(u_p)$. It is enough to prove that the reverse inequality holds, that is
\begin{equation}\label{eq:thesis}
I_p(u_p)\le c_p.
\end{equation}
Let $u_{\mathrm{GS}}$ be a $\mathcal C_p$-ground state solution of ($P_p$). By Lemma \ref{commonbound}, for $n$ large,
\[
u_{\mathrm{GS}}\le\left(\frac{q}{p}\right)^\frac{1}{q-p} \le \left(\frac{q}{p_n}\right)^\frac{1}{q-p_n}+1 =s_0(p_n),
\]
with $s_0$ as in \eqref{eq:s_0_def}.
Hence, by Lemma \ref{gH},  
\begin{equation}\label{eq:s0pn}
\lim_{n\to\infty} h_{p_n}(u_{\mathrm{GS}})=\lim_{n\to\infty}\frac{\|u_{\mathrm{GS}}\|_{W^{1,p_n}(B)}^{p_n}}{\|u_{\mathrm{GS}}\|^q_{L^q(B)}} = \frac{\|u_{\mathrm{GS}}\|_{W^{1,p}(B)}^{p}}{\|u_{\mathrm{GS}}\|^q_{L^q(B)}}=h_p(u_{\mathrm{GS}})=1,
\end{equation}
where we used that $u_{\mathrm{GS}}\in\mathcal N_p$. Therefore, thanks to \eqref{eq:s0pn}, we have
\begin{equation*}
\begin{aligned}
\lim_{n\to\infty} I_{p_n}&(h_{p_n}(u_{\mathrm{GS}})u_{\mathrm{GS}}) \\
& = \lim_{n\to\infty}  \frac{1}{p_n}(h_{p_n}(u_{\mathrm{GS}}))^{p_n}\|u_{\mathrm{GS}}\|^{p_n}_{W^{1,p_n}(B)}-\frac{1}{q}(h_{p_n}(u_{\mathrm{GS}}))^{q}\|u_{\mathrm{GS}}\|^{q}_
{L^q(B)}\\ 
& = I_p(u_{\mathrm{GS}}).
\end{aligned}
\end{equation*}
Altogether, 
\[
c_p=I_p(u_{\mathrm{GS}})=\lim_{n\to\infty}I_{p_n}(h_{p_n}(u_{\mathrm{GS}})u_{\mathrm{GS}})\ge \lim_{n\to\infty} I_{p_n}(u_{p_n})=I_p(u_p),
\]
where in the inequality we took into account that $u_{p_n}$ is a $\mathcal C_{p_n}$-ground state solution of ($P_{p_n}$) and the last equality is given by \eqref{eq:conv_en}. Thus, \eqref{eq:thesis} holds and the proof is concluded.
\end{proof}

\begin{proof}[Proof of Corollary \ref{cor:uptou2}]
Let $(u_{p_n})$ be a sequence as in the statement and consider a subsequence $(u_{p_{n_k}})$. By Theorem \ref{thm:upntou_p-GS}, $(u_{p_{n_k}})$ admits in turn a subsequence that converges to $u_2$, the unique $\mathcal C_2$-ground state solution of ($P_2$). Therefore, the full sequence $(u_{p_n})$ converges to $u_2$.
\end{proof}

\begin{proof}[Proof of Corollary \ref{cor:existence_m<2}]
Let $q>2+\lambda_2^{\mathrm{rad}}$, $(p_n)$ be any sequence such that $p_n\to 2^-$, and $u_{p_n}$ be a $\mathcal C_{p_n}$-ground state solution of ($P_{p_n}$) for every $n\in\mathbb N$. Then, by Corollary~\ref{cor:uptou2}, $(u_{p_n})$ converges uniformly in $\bar B$ to $u_2$, the $\mathcal C_2$-ground state solution of ($P_2$). Since $q>2+\lambda_2^{\mathrm{rad}}$, $u_2\not\equiv 1$, as proved in \cite{BNW}. Therefore, $u_{p_n}\not\equiv 1$ for $n$ large. By the arbitrariness of the sequence $(p_n)$, the conclusion follows.
\end{proof}

\begin{remark}
As already mentioned, if $1<m<2$ and $q$ is large, problem ($P_m$) has at least two nonconstant solutions in the cone $\mathcal C_m$ of radial, radially nondecreasing functions: the $\mathcal C_m$-ground state solution $u_m$ and a higher energy solution $v_m$, whose existence has been proved in \cite{BFGM}. Now, under the assumptions that $p_n,\,p\in(1,2)$, $p_n\to p$, and that there exists $q$ large enough for which the higher energy solution $v_{p_n}$ of ($P_{p_n}$) exists for every $n$ and $v_p$ exists for ($P_p$), a natural question is whether the limit function $v$, found in Theorem \ref{thm:upntou_p} for a subsequence of $(v_{p_n})$, coincides with the higher energy solution $v_p$ of ($P_p$). As a byproduct of the results in the present paper, we can guarantee that $v$ is not the $\mathcal C_p$-ground state solution $u_p$ of ($P_p$). Indeed, by Lemma~\ref{lem:conv-integrals-2}, we get that $I_{p_n}(v_{p_n})\to I_p(v)$ and  $I_{p_n}(1)\to I_p(1)$. Moreover, as proved in \cite{BFGM}, $I_{p_n}(v_{p_n})>I_{p_n}(1)$ for every $n$. Thus, passing to the limit in the previous inequality we get 
\[I_p(v)\ge I_p(1)>I_p(u_p),\]
where $u_p$ is the $\mathcal C_p$-ground state solution of ($P_p$).
\end{remark}
\section*{Acknowledgments}
\noindent All authors are members of INdAM.
F. Colasuonno was partially supported by the MUR-PRIN project 20227HX33Z funded by the European Union - Next Generation EU.
B. Noris was partially supported by the MUR-PRIN project 2022R537CS ``NO$^3$'' granted by the European Union -- Next Generation EU. 
E. Sovrano was partially supported by the PRIN 2022 project ``Modeling, Control and
Games through Partial Differential Equations'' (D53D23005620006) funded by the European Union - Next Generation EU.

%

\end{document}